\documentclass[letterpaper,12pt]{amsart}
\usepackage{amssymb}
\usepackage{amscd}
\usepackage{hyperref}

\newtheorem{theo}{Theorem}[section]
\newtheorem{lemma}[theo]{Lemma}

\newtheorem*{introtheo}{Theorem}
\theoremstyle{definition}
\newtheorem{defi}[theo]{Definition}
\newtheorem*{introdefi}{Definition}
\theoremstyle{remark}
\newtheorem{remark}[theo]{Remark}

\def \avg {\operatorname{avg}}
\def \instability {\operatorname{instability}}
\def \Aut {\operatorname{Aut}}
\def \Sym {\operatorname{Sym}}

\begin{document}

\title{Bounds for Matchings in Nonabelian Groups}
\author{Will Sawin}
\address{ETH Institute for Theoretical Studies
ETH Zurich
8092 Zurich
william.sawin@math.ethz.ch}
\thanks{The author was supported by Dr. Max R\"{o}ssler, the Walter Haefner Foundation and the ETH Zurich Foundation. The author was also in residence at the Mathematical Sciences Research Institute in Berkeley, California during the Spring 2017 semester and supported by the National Science Foundation under Grant No. DMS-1440140.}
\maketitle

\begin{abstract} We give upper bounds for triples of subsets of a finite group such that the triples of elements that multiply to $1$ form a perfect matching. Our bounds are the first to give exponential savings in powers of an arbitrary finite group. Previously, \cite{BCCGNSU} gave similar bounds in abelian groups of bounded exponent, and \cite{Petrov} gave exponential bounds in certain $p$-groups.  \end{abstract}

The key concept of this paper is as follows:

\begin{introdefi} A {\em multiplicative matching} in a group $G$ is a triple of subsets $S,T,U$ in $G$ such that the subset $M \subseteq S\times T\times U$ consisting of triples $(s,t,u)$ satisfying $stu=1$ is a perfect matching of $S \times T \times U$ - i.e. each element of $S$ is in exactly one triple in $M$, and the same for $T$ and $U$. The {\em cardinality} of a multiplicative matching is $|M|=|S|=|T|=|U|$.\end{introdefi}

(The name was inspired by the use of {\em additive matching} as an alternate name for the {\em tri-colored sum-free sets} of \cite{BCCGNSU} in \cite{Aaronson} and replaces the term {\em tri-colored product-free sets} of \cite{Petrov}.) 

This paper is concerned with the problem of finding strong upper bounds on multiplicative matchings in finite non-abelian groups $G$. As in \cite{BCCGNSU}, a proof of sufficiently strong bounds along these lines for a given family of groups would rule out the possibility of proving the matrix multiplication constant is $2$ by finding subsets of those groups satisfying the simultaneous triple product property.
 
 The methods of this paper are based on the formalism of slice rank, and hence they build on the breakthrough work of Croot, Lev, and Pach \cite{CLP}, its generalization due to Ellenberg and Gijswijt \cite{EG}, and its interpretation in terms of slice rank, due to Tao \cite{Tao}. The idea of applying the method to group algebras in particular is due to Petrov \cite{Petrov}.
 
For most interesting families of group, the bounds in this paper are not strong enough to rule out proving the matrix multiplication constant is $2$ using subsets of those groups with the simultaneous triple product property. However, the methods used could allow one to give better bounds, possibly including bounds strong enough to have applications to matrix multiplication, with a better understanding of the modular group algebra $\mathbb F_p[G]$. 

The organization of this paper is as follows:

In Section 1, we present first a quick argument that uses as a black box certain results from \cite{BCCGNSU} to give the following bound:

\begin{introtheo}[\ref{set}]Let $G$ be a nontrivial finite group. There exists a constant $\delta<1$ such that the any multiplicative matching in $G^n$ has size at most $\delta^n |G|^n$. \end{introtheo}

We do not state a precise value of $\delta$, because, though it could be made explicit, the bound we obtain by this method would be far from optimal. In Section 2, we present a more in-depth argument that gives the stronger bound:

\begin{introtheo}[\ref{main2}]Let $H$ be a nonabelian finite simple group. Let $G$ be a group containing $H^n$ as a normal subgroup. Then the cardinality of a multiplicative matching in $G$ is at most 

\[\left(1 - \left( \frac{2- \frac{3}{2^{2/3}} }{|H|}\right) \right)^n |G|. \] \end{introtheo}

This essentially generalizes the previous bound because, for any finite group $G$, there must be a normal subgroup isomorphic to $H^k$ for some finite simple group $H$, and then $G^n$ contains $H^{nk}$ as a normal subgroup, so as long as $H$ is nonabelian we may take $\delta =\left(1 - \left( \frac{2- \frac{3}{2^{2/3}} }{|H|}\right) \right)^k$. The case when $H$ is abelian was previously handled by Chris Umans (in personal communication) with an argument that gives a similar bound.

The method of this section is adapted in part from \cite{Petrov}, and the strategy used to bound multiplicative matchings in a group using only properties of a normal subgroup is based on that of Umans.

In this section we also state a bound (Theorem \ref{main}) in terms of certain filtrations of the modular group algebra $\mathbb F_p[H]$. In the special case of a filtration of $\mathbb F_p[H]$ by powers of a two-sided ideal $I$ of $\mathbb F_p[H]$ satisfying $I^k=0$, the bound of Theorem \ref{main} is an explicit function of the dimensions of $I, I^2, \dots, $ and $I^{k-1}$ that grows stronger as these dimensions grow larger but grows weaker as $k$ grows larger.

Hence a study of the group algebras of finite simple groups in various characteristics leading to the construction of filtrations by large ideals whose small powers vanish could give a significant improvement of the bound in Theorem \ref{main2}. 

Section 3 begins the program of finding improved filtrations by studying the finite simple group $PSL_2(\mathbb F_p)$. We obtain the following result:

\begin{introtheo}[\ref{linear}] Let $p>3$ be a prime and let $G$ be a group containing $PSL_2(\mathbb F_p)^n$ as a normal subgroup. The size of a multiplicative matching in $G$ is at most \[ \delta_p^n |G|\]

where \[\delta_p= \inf_{\lambda \in (0,1)} \left( \frac{1 + \lambda^3 +\lambda^6}{3\lambda^2} + \frac{2 \lambda^2 - 1 - \lambda^6}{(p+1)\lambda^2} +\frac{2-2\lambda}{p^2-1} \right)\]

satisfies $\lim_{p \to \infty} \delta_p \approx .919$.
 \end{introtheo}
 
 This bound for the density is stronger than the bound of Theorem \ref{main2} in this special case, because $\lim_{p \to \infty} \delta_p<1$, while the constant in Theorem \ref{main2} converges to $1$ as $|H|$ goes to $\infty$.

We hope this type of bound can be generalized to larger classes of finite simple groups. The classification of finite simple groups suggests a potential path to proving bounds for all finite simple by successively handling larger and larger families, but even without relying on the classification, bounds for a restricted class of groups might have applications. One natural next step to take would be to prove similar, or stonger, bounds in $PSL_2(\mathbb F_{p^n})$.

I would like to thank Eric Naslund, David Speyer, and Chris Umans for helpful conversations.

\section{Soft Argument}

For a finite-dimensional algebra $R$ over a field $k$, the multiplication tensor of $R$ is a tensor in $R^\vee \otimes R^\vee \otimes R$. If we fix a basis $e_1,\dots,e_n$ of $R$, we can write the tensor as a function of $3$ variables whose values are the structure constants expressing a product of two basis vectors as a linear combination of basis vectors. Following \cite{BCCGNSU}, we will always view tensors as functions in this way.

\begin{lemma}\label{instability} Let $R$ be a finite-dimensional algebra over a field $k$. \begin{enumerate}

\item If $R$ is not semisimple, then the multiplication tensor is unstable (in the sense of  \cite[Definition 4.2]{BCCGNSU}).

\item Let $I$ be an ideal of $R$ such that $I^2=0$. Then the instability (in the sense of \cite[Definition 4.4]{BCCGNSU}) of the multiplication tensor is at least $\frac{\dim_k I}{3 \dim_k R}$

\end{enumerate} \end{lemma}

\begin{proof} First note that (2) implies (1). Indeed, if $R$ is not semisimple, then the Jacobson radical of $R$ is a nontrivial nilpotent idea.  Let $I$ be the largest power of the Jacobson radical that is nontrivial. Then certainly $I^2=0$. Therefore by part (2), the instability of the multiplication tensor is positive, so the multiplication tensor is unstable.

Now let us prove (2). Let $n = \dim_k R$ and $m=\dim_k I$. Take a basis $e_1,\dots,e_n$ for $R$ such that $e_1,\dots, e_m$ all lie in $I$.  Define weights $u_i= v_i  = 1$ if $i \leq m$ and $0$ if $i>m$. Define the weight $w_k$ to be $0$ if $i \leq m$ and $1$ if $i>m$. 

Let $r_{abc}$ be the structure constants of the multiplication tensor, i.e. $e_a e_b = \sum_{k=1}^n r_{abc} e_c $. 

Observe that if $u_a+v_b+w_c > 1$, then $r_{abc}=0$. Indeed, suppose that more than one of the events $a \leq m$, $b \leq m$, $c>m$ occur; we will show that $r_{abc}=0$. If $a\leq m$ and $b \leq m$ then $e_a \in I$ and $e_b\in I$, so $e_ae_b=0$ because $I^2=0$, thus $r_{abc}=0$ . Otherwise, we must have $c> m$ and either $a\leq m$ or $b\leq m$. In this case, we have either $e_a\in I$ or $e_b\in I$, so either way $e_a e_b \in I$, thus $r_{abc}=0$ because $c>m$.

Therefore if $r_{abc}\neq 0$, then $u_a+v_b+w_c \leq 1$. Hence we may write the multiplication tensor in terms of the basis $e_1,\dots,e_n$ and the dual basis $e_1',\dots,e_n'$ as
\[ \sum_{a,b,c | u_a+v_b+w_c \leq 1}  r_{abc} e_a' e_b' e_c .\]
So we may take $R=1$ in \cite[Definition 4.4]{BCCGNSU}. Then because
\[u_{\avg} + v_{\avg}+ w_{\avg}= \frac{m}{n}+ \frac{m}{n} + \frac{n-m}{n} = 1 + \frac{m}{n}=R +\frac{m}{n} \]
and \[u_{\max}-u_{\inf}+v_{\max}-v_{\inf}+w_{\max}-w_{\inf}=1-0+1-0+1-0+1-0 = 3,\] we may take \[\epsilon= \frac{m}{3n}\] as desired.

\end{proof}

\begin{lemma}\label{rank}  Let $R$ be a finite-dimensional algebra over a field $k$. If $R$ is not semisimple, then there exists a constant $\delta<1$ such that the slice rank of the multiplication tensor of $R^{\otimes n}$ is at most $3\delta^n (\dim_k R)^n$. \end{lemma}

\begin{proof}  Let $F$ be the multiplication tensor of $R$. By \cite[Theorem 4.10]{BCCGNSU}, the slice rank of the tensor power $F^{\otimes n}$ is at most $3 (\dim_k R)^n e^{-2n \operatorname{instability}(F)^2}$. Letting $\delta= e^{ -2 \operatorname{instability}(F)^2}<1$ (as $\instability(F)>0$ by Lemma \ref{instability}), and observing that $F^{\otimes n}$ is the multiplication tensor of $R^{\otimes n}$, we get the stated result. \end{proof}

Define the multiplication tensor of a group to be the multiplication tensor of its group algebra - the function $F(s,t,u)$ defined for $s,t,u\in G$ that is $1$ if $st=u$ and $0$ otherwise.

\begin{lemma}\label{group} Let $G$ be a finite group, and let $p$ be a prime dividing $|G|$. There exists a constant $\delta<1$ such that the slice rank of the multiplication tensor of $G^n$ over $\mathbb F_p$ is at most $3 \delta^n |G|^n$. \end{lemma}

\begin{proof} Set $R = \mathbb F_p[G]$. By the converse to Maschke's theorem, $R$ is not semisimple. For instance, the vector space generated by the element $\sum_{g \in G} g$ is a nontrivial two-sided ideal that squares to $0$. Note that \[\mathbb F_p[G^n] = \mathbb F_p[G]^{\otimes n}=R^{\otimes n}\] and then apply Lemma \ref{rank}. \end{proof}

\begin{lemma}\label{tricolored} Let $H$ be a finite group. For any field $k$, the cardinality of a multiplicative matching in $H$ is at most the slice rank of the multiplication tensor of $H$ over $k$. \end{lemma}

\begin{proof} Consider the function $F$ on $H \times H \times H$ where $F(s,t,u)$ is $1$ if $stu=1$ and $0$ otherwise. This is the same as the multiplication tensor but with the last variable inverted, so it has the same slice rank. Let $ S\times T\times U$ be a multiplicative matching in $H$. The slice rank of $F$ is at least the slice rank of the restriction of $F$ to $S \times T \times U$, which is the diagonal tensor, whose slice rank is $|S|$ by \cite[Lemma 4.7]{BCCGNSU}. So the slice rank of the multiplication tensor is at least the cardinality of the multiplicative matching.\end{proof}

\begin{theo}\label{set} Let $G$ be a nontrivial finite group. There exists a constant $\delta<1$ such that any multiplicative matching in $G^n$ has size at most $\delta^n G^n$. \end{theo}

\begin{proof} Since $G$ is nontrivial, there exists a prime $p$ dividing $|G|$. By Lemma \ref{group}, the slice rank of the multiplication tensor of $G^n$ is at most $3 \delta^n |G|^n$. By Lemma \ref{tricolored}, the size of any multiplicative matching in $G^n$ is at most $3 \delta^n |G|^n$. We can remove the factor of $3$ by a standard amplification argument - a multiplicative matching in $G^n$ of size $|S|$ gives a multiplicative matching in $|G|^{nk}$ of size $|S|^k$, hence $|S|\leq 3^{1/k} \delta^n|G|^n$, and taking $k$ to $\infty$, $3^{1/k}$ goes to $0$.

\end{proof}

\section{Main Results}

We will now give a second bound for multiplicative matchings using a more detailed argument. This will apply to a power of a finite group $H^n$ or more generally to any group containing $H^n$ as a normal subgroup. The final bound will depend on the existence of certain filtrations of the group algebra $k[H]$ over a field $k$. First we have two lemmas that provide given the existence of single ideal in $k[H^n]$ with certain properties. We  will then construct the ideal using a filtration of $k[H]$.

The following lemma is a variant of \cite[Theorem 2]{Petrov}:

\begin{lemma}\label{slicing} Let $G$ be a finite group and $k$ a field. Let $K$ be a $k$-subspace of the group algebra $k[G]$ such that $K^3=0$. Then the slice rank of the function $1_{xyz=1}: G \times G \times G \to k$ is at most $3  (|G| - \dim K)$. \end{lemma}

\begin{proof} Let $n=|G|$ and $k= \dim K$. Let $e_1,\dots,e_k$ be a basis of $K$,  $e_{k+1},\dots,e_n$ an extension to a basis of $k[G]$. Let $e_1',\dots,e_n'$ be the dual basis of $e_1,\dots,e_n$. Then we may write the function $1_{xyz=1}$ as $\sum_{1 \leq a,b,c \leq n} r_{a,b,c} e_a'(x) e_b'(y) e_c'(z)$ where, by duality, $r_{a,b,c}$ must be  $ \sum_{x,y,z \in G} 1_{xyz=1} e_a(x) e_b(y) e_c(z)$. This is the same as the coefficient of $1$ in the element $(\sum_{x\in G} e_a(x) x) (\sum_{y \in G} e_b(y) y) (\sum_{z\in G} e_c(z) z )$ of the group algebra $k[G]$. If $a \leq k$, $b \leq k$, $c\leq k$, then that product is $0$, so that coefficient vanishes. Thus if $r_{a,b,c}\neq 0$, then either $a>k$, $b>k$, or $c>k$.

Using this, we may group the terms $r_{a,b,c} e_a'(x) e_b'(y) e_c'(z)$ with $r_{a,b,c}\neq 0$ into batches, where each batch is $e_a'(x)$ times a function of only $y$ and $z$, $e_b'(y)$ times a function of only $x$ and $z$, or $e_c' (z)$ times a function of only $x$ and $y$ for $a$, $b$, or $c$ greater than $k$. Each batch is a slice, so the slice rank is at most the number of batches, which is $3(n-k)$. \end{proof}

The following is a generalization of an argument of Chris Umans (in personal communication) from the special case $H =\mathbb F_p^n$:

\begin{lemma}\label{jumping} Let $G$ be a finite group, $H$ a normal subgroup, and $k$ a field. Let $J$ be a subspace of the group algebra $k[H]$ such that $J^3=0$ and that is invariant under the natural action of $G$ by conjugation on $H$. Then there exists a two-sided ideal $K$ of $k[G]$, of dimension at least $\frac{|G|}{|H|} \dim J$, satisfying $K^3=0$. \end{lemma}

\begin{proof} Let $K$ be the two-sided ideal of $k[G]$ generated by $J$. Because $k[G]$ is generated as a vector space by elements $g\in G$,  $K$ is generated as a vector space by the subspaces of the form $g_1 J g_2$ for $g_1,g_2 \in G$. To check that $K^3=0$, it is sufficient to check that for all $g_1,g_2,g_3,g_4\in G$,  $g_1 J g_2 J g_3 J g_4=0$. Because $J$ is invariant under conjugation by the $g_i$, it commutes with the $g_i$, so this is the same as $g_1 g_2 J^3 g_3 g_4$, and $J^3$ vanishes by assumption.

Because $k[G]$ is a free $k[H]$-module of rank $|G|/|H|$,  $J k[G]$ contains a copy of $J^{|G|/|H|}$ generated by a $k[H]$-basis of $k[G]$. Because $K$ contains $J k[G]$, which contains $J^{|G|/|H|}$, which has dimension $\frac{|G|}{|H|} \dim J$, the dimension of $K$ is at least $\frac{|G|}{|H|} \dim J$. \end{proof}

\begin{remark} In the case $H = G$, the argument of Lemma \ref{jumping} shows that if there is a subspace $J$ of $k[G]$ with $J^3=0$ that is conjugacy-invariant, there is a two-sided ideal, at least as large, that also cubes to zero and is conjugacy-invariant. So we may as well restrict our attention to the construction of two-sided ideals satisfying the conditions of Lemma \ref{jumping}. \end{remark}

 Let $H$ be a finite group and $k$ a field. Suppose we have a filtration of $k[H]$ indexed by the interval $[0,1]$, i.e. for each $\alpha \in [0,1]$ we have a two-sided ideal $I_\alpha$ such that if $\alpha \leq \beta$ then $I_\beta \subseteq I_\alpha$. Assume also that $I_0=k[H]$ and, for all $\alpha$, $\bigcap_{\beta<\alpha} I_\beta = I_\alpha$.
 
 Suppose  that, for all triples of $\alpha,\beta,\gamma \in [0,1]$ with $ \alpha+\beta+\gamma>1$ we have $I_\alpha I_\beta I_\gamma =0$. 
 
We will use this filtration to bound the size of multiplicative matchings in $H^n$ and, under additional assumptions, any group containing $H^n$ as a normal subgroup.

 \begin{defi}\label{jjj} Let $J$ be the $k$-vector subspace of $k[H^n] = k[ H]^{\otimes n}$ generated by $I_{\alpha_1} \otimes I_{\alpha_2} \otimes \dots \otimes I_{\alpha_n}$ for all tuples $(\alpha_1,\dots,\alpha_n) \in [0,1]^n$ satisfying $\sum_{i=1}^n \alpha_i > n/3$. \end{defi}
 
 Because $I_\alpha$ is a two-sided ideal for each $\alpha$, $I_{\alpha_1} \otimes I_{\alpha_2} \otimes \dots \otimes I_{\alpha_n}$ is a two-sided ideal for each tuple $(\alpha_1,\dots,\alpha_n)$, and thus $J$ is a two-sided ideal.

  \begin{lemma}\label{cubesnought2}We have \[J^3=0.\] \end{lemma}
 
 \begin{proof} It is sufficient to prove that any triple of $n$-tuples $(\alpha_1,\dots,\alpha_n), (\beta_1,\dots,\beta_n),(\gamma_1,\dots,\gamma_n)$, satisfying $\sum_{i=1}^n \alpha_i>n/3, \sum_{i=1}^n \beta_i>n/3,\sum_{i=1}^n \gamma_i > n/3$, we have \[(I_{\alpha_1} \otimes \dots \otimes I_{\alpha_n})(I_{\beta_1} \otimes \dots I_{\beta_n}) (I_{\gamma_1} \otimes \dots I_{\gamma_n}) =0 .\]
 
Observe that 

\[(I_{\alpha_1} \otimes \dots \otimes I_{\alpha_n})(I_{\beta_1} \otimes \dots \otimes  I_{\beta_n}) (I_{\gamma_1} \otimes \dots \otimes I_{\gamma_n}) = I_{\alpha_1}  I_{\beta_1} I_{\gamma_1} \otimes \dots \otimes I_{\alpha_n} I_{\beta_n} I_{\gamma_n}. \]

Under our assumptions, $\sum_{i=1}^n (\alpha_i+\beta_i+\gamma_i)>n$, so for some $i$ we have $\alpha_i+\beta_i+\gamma_i>1$, hence $I_{\alpha_i} I_{\beta_i} I_{\gamma_i}=0$ and this tensor product vanishes as desired.\end{proof}
 
   \begin{lemma}\label{invariant} Suppose that for each $\alpha$, $I_\alpha$ is invariant under the action of $\Aut(H)$ on $k[H]$.
   
   Then $J$ is invariant under the natural action of the wreath product $\Aut(H) \wr S_n$ on $H^n$. \end{lemma}
 
 \begin{proof} It is clear that $J$ is invariant under the action of $\Aut(H)$ on each factor because each $I_\alpha$ is invariant under the action of $\Aut(H)$. An element $\sigma$ of $S_n$ sends  $I_{\alpha_1} \otimes I_{\alpha_2} \otimes \dots \otimes I_{\alpha_n}$ to $I_{\alpha_{\sigma(1)}} \otimes I_{\alpha_{\sigma{2}}} \otimes \dots \otimes I_{\alpha_{\sigma(n)}}$ and so preserves $J$ as well.
 
 Hence $J$ is invariant under the wreath product, which is generated by these two subgroups.
 
 \end{proof}

   \begin{lemma}\label{dimensionformula} We have \[|H|^n - \dim J \leq  \left(  \inf_{\lambda \in (0,1)}\frac{|H| + \int_0^1 \dim (I_\alpha) \lambda^\alpha \log \lambda d\alpha }{\lambda^{1/3} } \right)^n,\] \end{lemma}
 
 \begin{proof} It suffices to prove for each fixed $\lambda, 0<\lambda<1$ that 
 
 \[|H|^n - \dim J \leq  \left( \frac{|H| + \int_0^1 \dim (I_\alpha) \lambda^\alpha \log \lambda d\alpha }{\lambda^{1/3} } \right)^n.\] 
 
 Let $m=|H|$. 
 
 We can choose a basis $e_1,\dots,e_m$ of $k[H]$ such that each $I_\alpha$ is generated by the first $\dim I_{\alpha}$ basis vectors. (Because $I_\beta \subseteq I_\alpha$ when $\beta>\alpha$ there can be only finitely distinct $I_\alpha$, all totally ordered by inclusion. Choose a basis of the smallest, then extend it to a basis of the next smallest, and so on.)
 
 For $j$ from $1$ to $m$, let $\alpha_j$ be the largest $\alpha$ such that  $e_j \in I_{\alpha_j}$. There exists such an $\alpha$ by the assumption $I_0=k[H]$, and the supremum is obtained by the assumption that $I_{\alpha} = \bigcap_{\beta<\alpha} I_\beta$.
 
 Any product of basis vectors $e_{j_1} \otimes \dots \otimes e_{j_n}$ such that $\sum_{i=1}^n \alpha_{j_m} > n/3$ is in $J$.
 
 So the set of all products of basis vectors $e_{j_1} \otimes \dots \otimes e_{j_n}$ such that $\sum_{i=1}^n \alpha_{j_m} \leq n/3$ generates $k[H^n]/J$. Therefore $|H|^n - \dim J$ is at most the number of such products of basis vectors.
 
 The number of such products is \[\sum_{j_1,\dots,j_n \in \{1,\dots m\} | \sum_{i=1}^n \alpha_{j_i} \leq n/3} 1.\] We have
 
 \[ \sum_{j_1,\dots,j_n \in \{1,\dots m\} | \sum_{i=1}^n \alpha_{j_i} \leq n/3} 1 \leq \sum_{j_1,\dots,j_n \in \{1,\dots m\} | \sum_{i=1}^n \alpha_{j_i} \leq n/3} \lambda^{\sum_{i=1}^n \alpha_{j_i} - n/3}  \leq \sum_{j_1,\dots,j_n \in \{1,\dots m\} } \lambda^{\sum_{i=1}^n \alpha_{j_i} - n/3} \]
 
 \[\leq \left(\sum_{j=1}^m  \lambda^{\alpha_j - 1/3} \right)^n.\]

 Finally
 
 \[ \sum_{j=1}^m \lambda^{\alpha_j} =\sum_{j=1}^m \left(1 + \int_0^{\alpha_j} \lambda^\alpha \log \lambda d \alpha\right)= m + \int_0^1 \left|\{j | \alpha_j \geq \alpha_m\}\right| \lambda^\alpha \log \lambda d\alpha  = m + \int_0^1 \dim (I_\alpha)  \lambda^\alpha \log \lambda d\alpha\]

and we obtain the stated formula.

 \end{proof}

 In practice, it will be convenient to give a slightly different formula for $|H| + \int_0^1 \dim (I_\alpha) \lambda^\alpha \log \lambda d\alpha $, which requires additional notation:
 
 \begin{lemma}\label{alternateformulation} Let $J_1, \dots, J_k$ be two-sided ideals and let $\alpha_0,\dots , \alpha_{k}$ be real numbers  such that if $j > i$ then $J_j \subseteq J_i$ and $\alpha_j > \alpha_i$. Also suppose that $\alpha_0 =0$ and $\alpha_k=1$.  Let $I_0= k[G]$ and $I_\alpha = J_{\min \{i | \alpha_i \geq \alpha\} }$ for $\alpha\in (0,1]$. Then
 
 \[ |H| + \int_0^1 \dim (I_\alpha) \lambda^\alpha \log \lambda d\alpha = \dim (k[H] / J_1) + \sum_{i=1}^{k-1} \dim (J_i / J_{i+1} ) \lambda^{\alpha_i} + \dim (J_k) \lambda \] 
 
 \end{lemma}
 
 \begin{remark}This formula is closely related to the rate function of \cite[Equation (2)]{EG} and to the entropy minimization of \cite[Proposition 6]{ST}.\end{remark}

 \begin{proof} \[ H = \dim (k[H] / J_1) + \sum_{i=1}^{k-1} \dim (J_i / J_{i+1} ) + \dim (J_k)  \] and

 \[\dim I_{\alpha} = \sum_{i=1}^{k-1}  1_{\alpha_i \geq \alpha} \dim (J_i / J_{i+1} )  + \dim (J_k)\]
 
 so
 
 \[   |H| + \int_0^1 \dim (I_\alpha) \lambda^\alpha \log \lambda d\alpha\]
 
 \[ = \dim (k[H] / J_1)  + \sum_{i=1}^{k-1}\dim(J_i/J_{i+1})  \left( 1+ \int_0^1 1_{\alpha_i \geq \alpha} \lambda^\alpha \log \lambda d\alpha \right) + \dim (J_k)\left(1+ \int_0^1 \lambda^\alpha \log \lambda d \alpha\right)\]
 
 \[=\dim (k[H] / J_1)  + \sum_{i=1}^{k-1}\dim(J_i/J_{i+1})  \left( 1+ \int_0^{\alpha_i}  \lambda^\alpha \log \lambda d\alpha \right) + \dim (J_k)\left(1+ \int_0^1 \lambda^\alpha \log \lambda d \alpha\right)\]

  \[=\dim (k[H] / J_1)  + \sum_{i=1}^{k-1}\dim(J_i/J_{i+1})  \left( 1+ \int_0^{\alpha_i} \frac{d \lambda^\alpha}{d \alpha} d\alpha \right) + \dim (J_k)\left(1+ \int_0^1  \frac{d \lambda^\alpha}{d \alpha} d \alpha\right)\]

  \[=\dim (k[H] / J_1)  + \sum_{k=1}^{a-1}\dim(J_i/J_{i+1}) \lambda^{\alpha_i} + \dim (J_k)\lambda.\]

 \end{proof}

\begin{theo}\label{main}  Let $H$ be a finite group. Let $k$ be a field and $I_\alpha$ a filtration of $k[H]$ into two-sided ideals, indexed by $[0,1]$, such that $I_0=k[H]$, $I_\alpha =\bigcap_{\beta<\alpha} I_\beta$, and $I_\alpha I_\beta I_\gamma =0$ whenever $\alpha+\beta+\gamma>1$.

\begin{enumerate}

\item The cardinality of a multiplicative matching in $H^n$ is at most \[  \left(  \inf_{\lambda \in (0,1)}\frac{|H| + \int_0^1 \dim (I_\alpha) \lambda^\alpha \log \lambda d\alpha }{\lambda^{1/3} } \right)^n\]

\item If we suppose furthermore that $H$ is an indecomposable group with trivial center and, for each $\alpha$, $I_\alpha$ is invariant under automorphisms of $H$, then for $G$ a finite group such that $H^n$ is a normal subgroup of $G$, the cardinality of a multiplicative matching in $G$ is at most 

\[ |G|  \left(  \inf_{\lambda \in (0,1)}\frac{1+ \int_0^1 \frac{\dim (I_\alpha)}{|H|} \lambda^\alpha \log \lambda d\alpha }{\lambda^{1/3} } \right)^n\]

\end{enumerate}

\end{theo}

\begin{proof} First we prove part (1). Take a multiplicative matching $ S \times T \times U$ in $H^n$. Restricted to $S \times T \times U$, the function $1_{xyz=1}$ becomes the characteristic function of the perfect matching $M$, whose slice rank is the cardinality $|S $. So the cardinality of the multiplicative matching is at most the slice rank of the function $1_{xyz=1}$ on $(H^n)^3$ over the field $k$.

Let $J$ be the two-sided ideal in $k[H^n]$ defined in Definition \ref{jjj}. By Lemma \ref{cubesnought2}, $J^3=0$. Hence by Lemma \ref{slicing} and Lemma \ref{dimensionformula}, the slice rank of the function $1_{xyz=1}$ on $(H^n)^3$ over the field $k$  is at most \[3 (|H|^n- \dim J) \leq 3  \left(  \inf_{\lambda \in (0,1)}\frac{|H| - \int_0^1 \dim (I_\alpha) \lambda^\alpha \log \lambda d\alpha }{\lambda^{1/3} } \right)^n.\]

We can remove the factor of $3$ by a standard amplification argument - a multiplicative matching in $|H|^n$ of cardinality $|S|$ gives a multiplicative matching in $|H|^{kn}$ of cardinality $|S|^k$, letting us improve the $3$ to $3^{1/k}$, which tends to $1$ in the limit as $k$ goes to $\infty$.

Next we prove part (2). Assume that $H$ is an indecomposable group with trivial center, $I_\alpha$ is $\Aut(H)$-invariant, and $H^n$ is a normal subgroup of $G$. Take a multiplicative matching $S \times T \times U$ in $G$. By the same argument, the cardinality of the multiplicative matching is at most the slice rank of the function $1_{xyz=1}$ on $G^3$ over $k$.

Let $J$ be the same two-sided ideal as before. By Lemma \ref{invariant}, $J$ is invariant under $\Aut(H) \wr S_n$. Because $H$ is indecomposable with trivial center, $\Aut(H^n) =\Aut(H) \wr S_n$ \cite[Theorem 3.1]{Bidwell}. So $J$ is invariant under the conjugation action of $G$.

Hence by Lemma \ref{jumping} there is a two-sided ideal $K$ of $k[G]$ with $\dim K= \frac{|G|}{|H|^n} \dim J$. Hence by Lemma \ref{slicing} and Lemma \ref{dimensionformula} the slice rank of $1_{xyz=1}$ is at most

\[3 \left(|G| - \frac{|G|}{|H|^n} \dim J\right)= 3 \frac{|G|}{|H|^n}  \left( |H|^n - \dim J\right) \]

\[\leq 3 \frac{|G|}{|H|^n} \ \left(  \inf_{\lambda \in (0,1)}\frac{|H| + \int_0^1 \dim (I_\alpha) \lambda^\alpha \log \lambda d\alpha }{\lambda^{1/3} } \right)^n= 3 |G|\left(  \inf_{\lambda \in (0,1)}\frac{1- \int_0^1 \frac{\dim (I_\alpha)}{|H|} \lambda^\alpha \log \lambda d\alpha }{\lambda^{1/3} } \right)^n .\] 

So the cardinality of a multiplicative matching is at most $3 |G|  \left(  \inf_{\lambda \in (0,1)}\frac{1+ \int_0^1 \frac{\dim (I_\alpha)}{|H|} \lambda^\alpha \log \lambda d\alpha }{\lambda^{1/3} } \right)^n$. We remove the factor of $3$ with the same amplification argument as in part (1) except that we pass to $G^k$, which admits $H^{nk}$ as a normal subgroup, to improve the bound.

\end{proof}

Let us consider a special case:

Let $I$ be a two-sided ideal of $k[H]$ that satisfies $I^2=0$.  Let $I^\perp$ be the ideal of elements $x \in k[H]$ such that $xI=0$ and $Ix=0$.

 \begin{lemma}\label{duality} $\dim I^\perp = |H| -\dim I $ \end{lemma}
 
 \begin{proof}Consider the bilinear form on $k[H]$ that sends two elements to the coefficient of $1$ in their product. From the definition of  multiplication in the group algebra, this is a perfect pairing: $g$ pairs nontrivially with $g^{-1}$ and no other element. So the perpendicular subspace of $I$ under this pairing has dimension $|H| -\dim I$. It is sufficient to show that this is $I^\perp$.
 
 Given an element $x \in k[G]$ whose pairing with $I$ vanishes, and $y\in I$, the coefficient of $1$ in $xy$ vanishes. For any $y \in I$, the coefficient of $g$ in $xy$ is the coefficient of $1$ in $xyg^{-1}$ and vanishes because $yg^{-1}\in I$. So $xy=0$ for $y \in I$.
 
 Similarly, with the same $x$ and $y$, the coefficient of $g$ in $yx$ is the coefficient of $1$ in $g^{-1}yx$. By examining the definition of multiplication in the group algebra we can see that the pairing is symmetric, so because $g^{-1}y \in I$, this coefficient vanishes.
 
 \end{proof}

\begin{theo}\label{main2} Let $H$ be an indecomposable group with trivial center. Let $G$ be a group containing $H^n$ as a normal subgroup. Then the cardinality of a multiplicative matching in $G$ is at most 

\[|G| \left(1 - \left( \frac{2- \frac{3}{2^{2/3}} }{|H|}\right) \right)^n .\]

Furthermore, if over some field $k$ there exists an ideal $I$ in the group algebra $k[H]$ satisfying $I^2=0$ and which is invariant under the group of automorphisms of $H$, then the cardinality of a multiplicative matching in $G$ is at most 

\[|G| \left(1 - \left( \frac{2- \frac{3}{2^{2/3}}}{|H|}\right)  \dim (I) \right)^n .\]

\end{theo}

\begin{proof}  The second statement implies the first, because for each prime $p$ dividing the cardinality of $H$, we may take $I$ to be the ideal generated by the element $\sum_{h \in H} h  \in \mathbb F_p[H]$, which is a one-dimensional ideal that squares to $0$.  So it is sufficient to prove only the second statement.

To do this, consider the filtration $I_\alpha$ on $k[H]$ given by $I_\alpha = k[H]$ if $\alpha=0$, $I_\alpha= I^\perp$ if $0<\alpha \leq 1/3$, and $I_\alpha= I$ if $1/3 < \alpha \leq 1$.  This is indeed a filtration because $I^2=0$ so $I \subseteq I^\perp$. 
We apply Theorem \ref{main}(2) to this filtration. 

First we check the conditions. By definition, $I_0=0$ and $I_\alpha= \bigcap_{\beta<\alpha} I_\beta$. Furthermore, if $\alpha+\beta+\gamma>1$ then one of the $\alpha,\beta,\gamma$ is  greater than $ 1/3$ and another is greater than $0$ so $I_\alpha I_\beta I_\gamma \subseteq I I^\perp =0$. Because $I$ is invariant under automorphisms of $H$, $I^\perp$ is as well, so every ideal in the filtration is automorphism-invariant. We have already assumed that $H$ is a nonabelian finite simple group and that $H^n$ is a normal subgroup of $G$. 

It remains to calculate \[ \inf_{\lambda, 0 < \lambda<1} \left(|H| + \int_0^1 \dim (I_\alpha) \lambda^\alpha \log \lambda d\alpha \right).\]

Lemma \ref{alternateformulation}, together with Lemma \ref{duality}, imply that \[|H| + \int_0^1 \dim (I_\alpha) \lambda^\alpha \log \lambda d\alpha =  \dim (I)  + (|H| - 2 \dim (I)) \lambda^{1/3} + \lambda \dim (I)\]\[ =\lambda^{1/3} H + (1- 2 \lambda^{1/3} + \lambda^{2/3} )\dim(I) \]
and
\[ \inf_{\lambda \in (0,1)} \left(\frac{|H| + \int_0^1 \dim (I_\alpha) \lambda^\alpha \log \lambda d\alpha }{\lambda^{1/3}}\right) = |H| +\inf_{\lambda \in (0,1)} (\lambda^{-1/3} - 2 + \lambda^{2/3}) \dim (I) \]
(taking $\lambda=1/2$)
\[ \leq |H| +(2^{1/3} - 2 + 2^{-2/3})\dim (I)  = |H| + \left( \frac{3}{2^{2/3}} - 2\right) \dim (I). \]
\end{proof}

A similar theorem (with a stronger bound) was previously proven by Chris Umans for abelian finite simple groups $\mathbb F_p$. For this reason, and because the bound of this theorem gets stronger the smaller the group is, it is primarily of interest for nonabelian finite simple groups.

\section{Matrix Groups}

There is no reason to believe that the bounds of Theorem \ref{main2} are sharp, or close to it. This suggests the problem of improving them by applying Theorem \ref{main} to a different filtration. In this section we do precisely this in the case $H=PSL_2(\mathbb F_p)$. 
 
Let $p>2$ be a prime. We will define a filtration on the group algebra $\overline{\mathbb F}_p[PSL_2(\mathbb F_p)]$ by first defining a filtration on the group algebra $\overline{\mathbb F}_p[GL_2(\mathbb F_p)]$. We will establish properties of this filtration using the classifications of the irreducible representations of $GL_2(\mathbb F_p)$ in characteristic $0$, the irreducible representations of $GL_2(\mathbb F_p)$ in characteristic $p$, and the reductions modulo $p$ of the characteristic $0$ representations. See \cite{Prasad} for an overview of these topics.

Recall that all irreducible mod $p$ representations of $GL_2(\mathbb F_p)$ have the form $(\det)^j \otimes \Sym^i (V)$, for $j$ from $0$ to $p-1$ and $i$ from $0$ to $p-1$, where $V$ is the standard representation of $GL_2(\mathbb F_p)$ \cite[Lemma 1.3]{Prasad} .  The dual of $(\det)^j \otimes \Sym^{i}(V)$ is $(\det)^{p-1-j} \otimes \Sym^{i}(V)$ so the set of irreducible representations involving a given symmetric power $i$ is closed under duality.

\begin{defi}Let $I_1$ be the ideal of $\overline{\mathbb F}_p[GL_2(\mathbb F_p)]$ consisting of all elements that act trivially on all irreducible mod $p$ representations of $GL_2(\mathbb F_p)$  except $(\det)^j \otimes \Sym^{p-1} (V)$. 

Let $I_2$ be the ideal consisting of elements that act trivially on all irreducible mod $p$ representations. 

Let $I_3$ be the ideal generated by all elements $\sum_g f(g) g$ where $f$ is a matrix coefficient of an irreducible mod $p$ representation other than $(\det)^j \otimes \Sym^{p-1} (V)$.\end{defi}

\begin{lemma}\label{filtration} We have the inclusions \[I_3 \subseteq I_2 \subseteq I_1.\]\end{lemma}

\begin{proof} The inequality $I_2 \subseteq I_1$ follows from the definition. To check $I_3 \subseteq I_2$, it suffices to check that for $f(g)$ a matrix coefficient of an irreducible representation other than $(\det)^j \otimes \Sym^{p-1} (V)$, and $f'(g)$ a matrix coefficient of an arbitrary irreducible representation $W$, we have $\sum_g f(g) f'(g) =0$, since knowing this for every coefficient of $W$ implies that $\sum_g f(g)$ acts trivially on $W$.

Note that $f(g)$ is a power of $\det$ times a polynomial of degree $<p-1$ in the coefficients of the matrix, while $f'(g)$ is a power of $\det$ times a polynomial of degree $\leq p-1$ in the coefficients of the matrix, so their product is a power of $\det$ times a polynomial of degree $<2(p-1)$ in the coefficients of the matrix.

We can replace the power of $\det$ by a power between $1$ and $p-1$ that is congruent to it modulo $p-1$ without affecting the sum. Because a positive power of $\det$ vanishes on matrices of determinant $0$, we may as well view the sum as a sum over all matrices instead of just invertible ones.

The power of $\det$ can be represented by a polynomial of degree $\leq 2(p-1)$ in the coefficients of the matrix. The final sum is the sum over $\mathbb F_p^4$ of a polynomial of degree $<4(p-1)$ and thus vanishes, as any monomial whose sum over $\mathbb F_p^4$ is nonvanishing must have degree at least $p-1$ in each variable and hence total degree at least $4(p-1)$ (as in the proof of Chevalley-Warning). \end{proof}

\begin{lemma}\label{easyproduct}We have the identity \[I_1 I_3 =0 .\] \end{lemma}

\begin{proof}  This follows from the definition. If $\sum_g f(g) g$ is in $I_1$, and $\langle v_1, g' v_2\rangle$ is a matrix coefficient of an irreducible representation $W$ other than $(\det)^j \otimes \Sym^{p-1}$ and so $\sum_{g'} \langle v_1, g'v_2\rangle g'$ in $I_3$, then \[\left(\sum_{h} f(h)h\right) \left(\sum_{g'} \langle v_1, g'v_2\rangle g'\right) = \sum_{g} \left(\sum_{h} f(h)\langle v_1, h^{-1} g v_2\rangle \right)g\] (where $g = hg'$) but $\sum_{h} f(h)\langle v_1, h^{-1} g v_2\rangle = \sum_{h} f(h)\left\langle\left (h^{-1}\right)^{T} v_1,  g v_2\right\rangle$ is a matrix coefficient of the action of $\sum_h f(h) \left (h^{-1}\right)^{T}$ on $W$, hence a matrix coefficient of the action of $\sum_h f(h) h$ on $W^\vee$, which vanishes by the assumption that $\sum_h f(h) h = \sum_g f(g) g$ is in $I_1$.

\end{proof}

\begin{lemma}\label{charactercriterion} Let $M_1 , M_2, M_3$ be finite-dimensional bimodules for a finite-dimensional algebra $R$ over an algebraically closed field. Let $f: M_1 \otimes_R M_2 \to M_3$ be a homomorphism of $R$-bimodules.

If $f$ is nontrivial, then there must be irreducible left $R$-modules $V_1,V_2,V_3$ such that $V_1 \otimes V_2^\vee$ is a Jordan-Holder factor of $M_1$, $V_2 \otimes V_3^\vee$ is a Jordan-Holder factor of $M_2$, and $V_1 \otimes V_3^\vee$ is a Jordan-Holder factor of $M_3$. 

Here we interpret $V_1 \otimes V_2^\vee$ as a bimodule where the left $R$ acts on $V_1$ and the right $R$ acts on $V_2^\vee$, and similarly for the other tensor products.  \end{lemma}

\begin{proof} If $M_1,M_2,$ and $M_3$ are all irreducible, then we must have $M_1 =W_1 \otimes W_2^\vee$, $M_2 = W_3 \otimes W_4^\vee$, $M_3 = W_5 \otimes W_6^\vee$ for $W_1,W_2,W_3,W_4,W_5,W_6$ irreducible left $R$-modules. But $(W_1 \otimes W_2^\vee) \otimes_R (W_3 \otimes W_4^\vee) =0$ unless 
$W_3 = W_2$, in which case it is $W_1 \otimes W_4^\vee$.  Hence for there to be a nontrivial map from $(W_1 \times W_2^\vee) \otimes_R (W_3 \otimes W_4^\vee) $ to $W_5 \otimes W_6^\vee$, we must have $W_1=W_5, W_4=W_6, W_3= W_2$. So we may take $V_1=W_1,V_2=W_2,V_3=W_6$.

Now we show that if $M_1, M_2,M_3$ are not all irreducible and admit a nontrivial map $f: M_1 \otimes_R M_2 \to M_3$, then there are subquotients $M_1', M_2', M_3'$ of $M_1,M_2,M_3$ and a nontrivial map $f' : M_1' \otimes_R M_2' \to M_3'$.  This implies the desired result by induction on the length of $M_1$ plus the length of $M_2$ plus the length of $M_3$.

If $M_1$ is not irreducible, let $N_1$ be any submodule of $M_1$. If $f$ remains nonzero when composed with the map $N_1 \otimes_R M_2 \to M_1 \otimes_R M_2$, we take $M_1'=N_1$, $M_2' = M_2$, $M_3'=M_3$, and $f'$ this composition. Otherwise, $f$ descends to the quotient $(M_1/N_1) \otimes_R M_2 \to M_3$, so we take $M_1'=M_1/N_1$, $M_2'=M_2, M_3'=M_3$, and $f'$ this descent.

If $M_2$ is not irreducible, we perform the same argument but with $M_1$ and $M_2$ switched.

If $M_3$ is not irreducible, and $f$ is not surjective, we let $M_3'$ be the image of $f$. Otherwise, we let $M_3'$ be a quotient of $M_3$ by any nontrivial proper submodule and $f'$ the composition of $f$ with the projection map. Because $f'$ is surjective, and $M_3'$ is nontrivial, $f'$ is nontrivial.

\end{proof}

\begin{lemma}\label{jordanholder}

The Jordan-Holder factions of $R/I_1$ are $\left((\det)^j \otimes \Sym^{i} (V)\right) \otimes \left((\det)^j \otimes \Sym^{i} (V)\right)^\vee$ for  $i$ and $j$ from $0$ to $p-2$.

The Jordan-Holder factors of $I_1/I_2$ are $\left((\det)^j \otimes \Sym^{p-1} (V)\right) \otimes \left((\det)^j \otimes \Sym^{p-1} (V)\right)^\vee$ for $j$ from $0$ to $p-2$.

 The Jordan-Holder factors of $I_2/I_3$ are

\begin{itemize}

\item $\left(\det^{i_1} \otimes \Sym^{i_2 -i_1} (V)\right) \otimes \left( \det^{i_2} \otimes \Sym^{p-1 - i_2 - i_1}\right)^\vee$ for $ 0\leq i_1< i_2 \leq p-2$.

\item $ \left( \det^{i_2} \otimes \Sym^{p-1 - i_2 - i_1}\right) \otimes \left(\det^{i_1} \otimes \Sym^{i_2 -i_1} (V)\right) ^\vee$ for $ 0\leq i_1< i_2 \leq p-2$.

\item $\left( \det^{b+1} \otimes \Sym^{a-b-2} (V) \right) \otimes \left( \det^a \Sym^{p-1-(a-b)} (V) \right)^\vee$ for $0\leq b<a \leq p-1$

\item $ \left( \det^a \Sym^{p-1-(a-b)} (V) \right)\otimes \left( \det^{b+1} \otimes \Sym^{a-b-2} (V) \right) ^\vee$ for $0\leq b<a \leq p-1$

\end{itemize}

The Jordan-Holder factors of $I_3$ are $\left((\det)^j \otimes \Sym^{i} (V)\right) \otimes \left((\det)^j \otimes \Sym^{i} (V)\right)^\vee$ for  $i$ and $j$ from $0$ to $p-2$, each with multiplicity $1$. 

(These are descriptions as multisets, so the multiplicity of each factor is the number of times it appears on the list - usually one.)

\end{lemma}

\begin{proof} The first two statements of this lemma follow from the fact that $R/I_2$ is the product of the matrix algebras of all irreducible representations of $GL_2(\overline{\mathbb F}_p)$ and $R/I_1$ is the product of the matrix algebras of all irreducible representations of $GL_2(\overline{\mathbb F}_p)$ other than $(\det)^j \otimes \Sym^{p-1}$. Hence in particular $I_2/I_3$ is the product of all the matrix algebras of irreducible representations that are of the form $\det^j \otimes \Sym^{p-1}$. The matrix algebra of an irreducible representation $W$, as a bimodule, is precisely $W \otimes W^{\vee}$.

Similarly, $I_3$ is the sum inside $R$ of, for each representation $W= \left((\det)^j \otimes \Sym^{i} (V)\right) $ with $i$ and $j$ from $0$ to $p-2$, the space of matrix coefficients of $W$. The space of matrix coefficients is isomorphic to $W \otimes W^{\vee}$ as a bimodule. Because these spaces for different $W$ are non-isomorphic irreducible bimodules, $I_3$ is also the sum of these spaces as an abstract bimodule. Hence its Jordan-Holder factors are exactly $W \otimes W^{\vee}$ for all these representations. This implies the last statement of this lemma.

We will establish the third statement of this lemma, on the Jordan-Holder factors of $I_2/I_3$, by calculating the multiset of Jordan-Holder factors of $R$ and subtracting off the already calculated Jordan-Holder factors of $R/I_1, I_1/I_2$ and $I_3$.

Observe that the multiset of Jordan-Holder factors of $R$ is the union with multiplicity over each irreducible characteristic $0$ representation $V$ of $GL_2(\mathbb F_p)$ of the multiset of Jordan-Holder factors of $V \otimes V^\vee$. This is because the Jordan-Holder factors of a characteristic $0$ representation, mod $p$, depend only on the isomorphism class of the characteristic $0$ representation (see \cite[Theorem 32]{Serre}), and in characteristic $0$, the group algebra of any group is the sum of the matrix algebras of its irreducible representations.

The Steinberg representations in characteristic $0$ are parameterized by characters of the determinant. They reduce $\mod p$ to the representations $(\det)^j \otimes \Sym^{p-1}(V)$ for $j$ from $0$ to $p-2$ \cite[pp. 6]{Prasad}. Hence their contribution to the Jordan-Holder decomposition of $R$ is exactly the same as that of $I_2/I_3$.

The $p-1$ representations that factor through the determinant reduce mod $p$ to $(\det)^j \otimes \Sym^{0}(V)$ \cite[pp. 6]{Prasad}.

The principal series representations are parameterized by unordered pairs of distinct characters, which mod $p$ can be represented by unordered pairs of distinct integers from $0$ to $p-2$, or equivalently by pairs of integers with $i_1<i_2$. The Jordan-Holder factors of the reduction mod $p$ are $\det^{i_1} \otimes \Sym^{i_2 -i_1} (V)$ and $\det^{i_2} \otimes \Sym^{p-1 - i_2 - i_1}$ \cite[Lemma 4.1]{Prasad}. 

The representation $\det^i \otimes \Sym^j$ appears in exactly $1$ principal series if $j$ lies between $1$ and $p-2$ and $0$ otherwise - indeed if it appears we must have $i_1 =i $ and $i_2 =i+j$ or $i_1= p-1-i-j$, $i_2 = i$, and exactly one of those has $0 \leq  i_1<i_2 \leq p-2$ in that range.

The discrete series representations are parametrized by characters of $\mathbb F_{p^2}$ up to conjugation, which mod $p$ can be represented as $x \mapsto x^a \overline{x}^b$ where $a,b$ are distinct integers from $0$ to $p-1$. Conjugation corresponds to switching $a$ and $b$, so again we may assume $a>b$. In this case, the Jordan-Holder factors are $\det^{b+1} \otimes \Sym^{a-b-2} (V) $ and $\det^a \Sym^{p-1-(a-b)}(V)$ \cite[Lemma 4.2]{Prasad}. 

The representation $\det^j \otimes \Sym^i(V)$ appears in exactly $1$ discrete series if $i$ lies between $0$ and $p-2$ and $0$ otherwise.  Indeed, if $j>0$ and $i+j <p-1$, which implies in particular that $i\leq p-2$, then we may take $a=i+j+1$, $b=j-1$. We have $p-1 \geq a > b \geq 0$ and  $\det^{b+1} \otimes \Sym^{a-b-2} (V) =\det^j \otimes \Sym^i(V)$.  If $i+j \geq p-1$ and $i \leq p-2$, which implies in particular that $j=0$, then we may take $a=j$, $b=i+j+1-p$, and we have $p-1 \geq a > b \geq 0$ and $\det^a \Sym^{p-1-(a-b)}(V)=\det^j \otimes \Sym^i(V)$. Finally if $j=0$ and $i \leq p-2$ then we may take $a=p-1, b=j$ and we have $p-1 \geq a > b \geq 0$ and $\det^a \Sym^{p-1-(a-b)}(V)=\det^j \otimes \Sym^i(V)$. It is easy to see that these formulas give the only possible solutions.

So combining all non-Steinberg representations, each irreducible modular representation $W$ of $GL_2(\mathbb F_p)$, other than the $\det^i \otimes \Sym^{p-1}$, appears with multiplicity exactly $2$. This means that $W \otimes W^{\vee}$ appears with multiplicity exactly two in $\overline{\mathbb F}_p[GL_2(\mathbb F_p)]$, which exactly cancels $R/I_1$ and $I_3$. So the non-diagonal Jordan-Holder factors of $V \otimes V^{\vee}$ for $V$ an irreducible characteristic zero representation other than $W \otimes W^{\vee}$ for $W$ an irreducible characteristic zero representation are exactly those that appear in $I_2/I_3$. The previously discussed formulas of \cite[Lemmas 4.1 and 4.2]{Prasad} demonstrate that these are exactly the factors described in the statement of this lemma, with the first two coming from principal series and the last two coming from discrete series.

\end{proof}

\begin{lemma}\label{hardproduct} \[I_1 I_2 \subseteq I_3\] \end{lemma}

\begin{proof}  Let $R= \overline{\mathbb F}_p[GL_2(\mathbb F_p)]$. We know that the multiplication map $R \otimes_R R \to R$ restricts to a map $I_1 \otimes_R I_2 \to I_2$. We need to show the quotient map $I_1 \otimes_R I_2 \to I_2/I_3$ vanishes. It suffices to show that the map it factors through, $I_1/I_3 \otimes_R I_2/I_3 \to I_2/I_3$ vanishes. We apply Lemma \ref{charactercriterion} to $M_1 = I_1 /I_3$, $M_2 = I_2/I_3$, $M_3 =I_2/I_3$

Note first that the Jordan-Holder factors of $I_1 / I_2$ only involve representations with $j=p-1$, while the Jordan-Holder factors of $I_2 / I_3$ never involve those representations. Because $j$ is preserved by duality, this means that any triples as in Lemma \ref{charactercriterion} between the Jordan-Holder factors of $I_1/I_3$, $I_2/I_3$, and $I_2/I_3$ must in fact only involve Jordan-Holder factors of $I_2/I_3$.

So it is sufficient to check there is no $V_1, V_2, V_3$ with $V_1 \otimes V_2^{\vee}$, $V_2\otimes V_3^{\vee}$, and $V_1 \otimes V_3^\vee$ all Jordan-Holder factors of $I_2/I_3$. Using Lemma \ref{jordanholder}, we can express this statement graph-theoretically:

Consider a graph with vertices the irreducible mod $p$ representations of $GL_2(\mathbb F_p)$ and with:

\begin{enumerate}

\item One edge connecting $\det^{i_1} \otimes \Sym^{i_2 -i_1} (V)$ to $\det^{i_2} \otimes \Sym^{p-1 - i_2 - i_1}$ for each pair of integers $i_1,i_2$ with $0 \leq i_1 < i_2 \leq p-2$.

\item One edge connecting $\det^{b+1} \otimes \Sym^{a-b-2} (V)$ to $ \det^a \Sym^{p-1-(a-b)} (V)$ for each pair of integers $a,b$ with $0 \leq b < a \leq p-1$.

\end{enumerate}

By construction and Lemma \ref{jordanholder}, $V_1 \otimes V_2^\vee$ is a Jordan-Holder factor of $I_2/I_3$ if and only if there is an edge connecting $V_1$ to $V_2$ in this gaph. By Lemma \ref{charactercriterion}, it is sufficient to show that this graph has no $3$-cycles.

 If we color the first type of edges red and the second type blue, then each vertex is on at most one red edge and at most one blue edge. In any 3-cycle, two adjacent edges are necessarily the same color. Because they are the same color and touch the same vertex, they must be the same edge. Hence the third edge of the three-cycle must be an edge from a vertex to itself, i.e. a loop. Therefore it is sufficient to show that this graph has no loops.
 
 This is trivial to check from the construction of the edges: If the two vertices were equal for an edge of the first type, we would have $i_1=i_2$, which is impossible. If the two vertices were equal for an edge of the second type, we would have $b+1 = a$, which implies $a-b=1$, and $p-1- (a-b) = (a-b-2)$, which implies $a-b= (p+1)/2$, and these contradict each other as $(p+1)/2\neq 1$.

\end{proof}

We can embed $\overline{\mathbb F}_p[PSL_2(\mathbb F_p)]$ into $\overline{\mathbb F}_p[SL_2(\mathbb F_p)]$ by sending each element to the sum of its inverse images, divided by $2$ (using $p>2$). We can then embed $\overline{\mathbb F}_p[SL_2(\mathbb F_p)]$ into $\overline{\mathbb F}_p[GL_2(\mathbb F_p)]$ directly. 

Let $J_1,J_2,J_3$ be the inverse images of $I_1, I_2$, $I_3$ under this map. 

Then

\begin{lemma}\label{verification} We have the inclusions and identities \begin{enumerate} \item \[J_3 \subseteq J_2 \subseteq J_1\] \item  \[J_1J_3=0\]\item \[J_1J_2 \subseteq J_3\] \end{enumerate}\end{lemma}

\begin{proof}These follow immediately by restriction from Lemmas \ref{filtration}, \ref{easyproduct}, and \ref{hardproduct} respectively. \end{proof}

\begin{lemma}  We have the identities \[ \dim(\overline{\mathbb F}_p[PSL_2(\mathbb F_p)]/J_1) = \frac{p(p-1)(p-2)}{6}\]   \[ \dim J_1/J_2 = p^2 \]     \[ \dim J_2/J_3 =\frac{p (p^2-7)}{6} \]    \[ \dim J_3  =  \frac{p(p-1)(p-2)}{6}\]
\end{lemma}

\begin{proof} $\dim(\overline{\mathbb F}_p[PSL_2(\mathbb F_p)]/J_2)$ is the dimension of the image of $\overline{\mathbb F}_p[PSL_2(\mathbb F_p)]$ inside $\overline{\mathbb F}_p[GL_2(\mathbb F_p)]/I_1$, which is the product of the matrix algebras of all irreducible representations of $GL_2(\mathbb F_p)$. These representations have the form $\det^j \otimes \Sym^i V$ for some $i$ and $j$. For $i$ odd, the element $-1 \in GL_2(\mathbb F_p)$ acts by $-1$ on this representation. Since all elements in the image of $\mathbb F_p[PSL_2(\mathbb F_p)]$ are equal to their own composition with $-1$, this implies the image of $\mathbb F_p[PSL_2(\mathbb F_p)]$ vanishes on these representations.

Furthermore, the homomorphism from  $\overline{\mathbb F}_p[PSL_2(\mathbb F_p)]$ to the matrix algebra of $\det^j \otimes \Sym^i (V)$ is independent of $j$. So we reduce to a homomorphism from $\mathbb F_p[PSL_2(\mathbb F_p)]$ to the product of the matrix algebras $M_{i+1}$ over all even numbers $i$ from $0$ to $p-1$.  Since the representations $\Sym^i (V)$ for $i$ even remain irreducible and pairwise nonisomorphic as representations of $PSL_2(\mathbb F_p)$, this homomorphism is surjective, and the dimension of the image is $\sum_{k=0}^{(p-1)/2} (2k+1)^2$, so $\dim (\mathbb F_p[PSL_2(\mathbb F_p)]/J_2) = \sum_{k=0}^{(p-1)/2} (2k+1)^2$.

For $\mathbb F_p[PSL_2(\mathbb F_p)]/J_1$, everything is the same except that we do not count irreducible representations with $i=p-1$, so instead the largest possible value of $i$ is $p-3$, and thus we have \[ \dim (\mathbb F_p[PSL_2(\mathbb F_p)]/J_1) = \sum_{k=0}^{(p-3)/2} (2k+1)^2 = \frac{p(p-1)(p-2)}{6}.\]

Furthermore \[\dim J_1/J_2= \dim (\mathbb F_p[PSL_2(\mathbb F_p)]/J_2)  - \dim (\mathbb F_p[PSL_2(\mathbb F_p)]/J_1) = \sum_{k=0}^{(p-1)/2} (2k+1)^2 - \sum_{k=0}^{(p-3)/2} (2k+1)^2 = p^2.\]

All elements of $I_3$ are linear combinations of matrix coefficients of irreducible representations of $GL_2(\mathbb F_p)$ (other then $\det^i \otimes \Sym^{p-1}(V))$. Elements of $J_3$ are those matrix coefficients that are supported on $SL_2(\mathbb F_p)$ and are invariant under the action of multiplication by $-1$ on $SL_2(\mathbb F_p)$.

Certainly all elements of $J_3$ are linear combinations of matrix coefficients of irreducible representations of $SL_2(\mathbb F_p)$ other than $\Sym^{p-1} (V)$. Because all elements of $J_3$ are invariant under that involution, the contributions to the linear combination from any representation that does not descend to $PSL_2(\mathbb F_p)$ is equal to its own negation and so vanishes, so all elements of $J_3$ are matrix coefficients of irreducible representations of $PSL_2(\mathbb F_p)$ other than $\Sym^{p-1}(V)$. These irreducible representations are precisely $\Sym^{2k}(V)$ for $k$ from $0$ to $(p-3)/2$. 

Let us show that all matrix coefficients of these representations are actually elements of $J_3$. Let $\rho$ be the action of $GL_2(\mathbb F_p)$ on $\Sym^{2k}(V)$, let $v_2$ be a vector in $\Sym^{2k}(V)$, and let $v_1$ be a vector in  $\Sym^{2k}(V)^\vee$. Then for $g \in SL_2(\mathbb F_p)$, we have \[\langle v_1, \rho(g) v_2 \rangle = \frac{1}{p-1} \sum_{j=0}^{p-2}\langle v_1 , \det(g)^{j} \rho(g) \rangle\] and for all $g\not \in SL_2(\mathbb F_p)$, we have \[\frac{1}{p-1} \sum_{j=0}^{p-2}\langle v_1 , \det(g)^{j} \rho(g) \rangle =0\] So this sum, being a linear combination of matrix coefficient of irreducible representations of $GL_2(\mathbb F_p)$, is clearly an element of $I_3$, and also lies in the image of $\overline{\mathbb F}_p[PSL_2(\mathbb F_p)]$.  So $J_3$ indeed consists of the matrix coefficients of $\Sym^{2k}(V)$ for $k$ from $0$ to $(p-3)$ and hence has dimension the sum of the squares of the dimensions of these representations, as desired.

Finally as
 \[\dim \overline{\mathbb F}_p[PSL_2(\mathbb F_p)]= |PSL_2(\mathbb F_p)|= \frac{p(p+1)(p-1)}{2}\]
 
 and
 
 \[ \dim \overline{\mathbb F}_p[PSL_2(\mathbb F_p)]= \dim(\overline{\mathbb F}_p[PSL_2(\mathbb F_p)]/J_1) +  \dim J_1/J_2 + \dim J_2 / J_3 + \dim J_3 \]\[= \frac{p(p-1)(p-2)}{6} + p^2 + \dim J_2/ J_3 + \frac{p(p-1)(p-2)}{6}\]
 
 we have
 
 \[ \dim J_2/J_3 = \frac{p(p+1)(p-1)}{2}- 2 \frac{p(p-1)(p-2)}{6} - p^2 = \frac{ p (p^2-7)}{6} \]

\end{proof}

\begin{remark}The last step of the proof of Lemma \ref{hardproduct}, that the graph has no loops, is false in general for the analogous graph in $\mathbb F_p[PSL_2(\mathbb F_p)]$, which is why we did not work directly in that group algebra.\end{remark}

\begin{theo}\label{linear} Let $G$ be a group containing $PSL_2(\mathbb F_p)^n$ as a normal subgroup. The size of a multiplicative matching in $G$ is at most \[ |G| \inf_{\lambda \in (0,1)} \left( \frac{1 + \lambda^3 +\lambda^6}{3\lambda^2} + \frac{2 \lambda^2 - 1 - \lambda^6}{(p+1)\lambda^2} +\frac{2-2\lambda}{p^2-1} \right)^n   .\]
 \end{theo}
 
 \begin{proof} Set $I_0= \overline{\mathbb F}_p[PSL_2(\mathbb F_p)]$, $I_\alpha = J_1$ for $0<\alpha\leq 1/3$,  $I_\alpha = J_2$ for $1/3 < \alpha \leq 1/2$, and $J_3$ for $\alpha>1/2$.  This is actually a filtration by Lemma \ref{verification}(1), and certainly we have $I_0=\overline{\mathbb F}_p[PSL_2(\mathbb F_p)]$ and $I_\alpha = \bigcap_{\beta<\alpha} I_\beta$.
 
 Moreover, if $\alpha+ \beta + \gamma >1$ for $\alpha,\beta,\gamma \in[0,1]$, then either $\alpha,\beta,\gamma$ are all positive and one of $\alpha,\beta,\gamma$ is greater than $1/3$, so $I_\alpha I_\beta I_\gamma \subseteq J_1 J_1 J_2 \subseteq J_1J_3 = 0$ by Lemma \ref{verification}(3) and Lemma \ref{verification}(2), or  one of  $\alpha,\beta$ or $\gamma$ is zero, one is positive, and one is greater than $1/2$, so $I_\alpha I_\beta I_\gamma \subseteq J_1 J_3=0$ by Lemma \ref{verification}(2), .
 
 These ideals are automorphism-invariant because they are pullbacks of ideals of $\overline{\mathbb F}_p[GL_2(\mathbb F_p)]$, invariant under the conjugation action of $PGL_2(\mathbb F_p)$, which is also the automorphism group of $PSL_2(\mathbb F_p)$. 
 
 Applying Theorem \ref{main} and Lemma \ref{alternateformulation} we immediately get an upper bound of \[ |G| \left( \inf_{0<\lambda<1 }  \frac{  \frac{p(p-1)(p-2)}{6} + p^2 \lambda^{1/3} +  \frac{ p (p^2-7)}{6} \lambda^{1/2} +\frac{p(p-1)(p-2)}{6} \lambda  }{   \frac{p(p+1)(p-1)}{2}  \lambda^{1/3}} \right)^n\] for the size of a multiplicative matching in $G$. Substituting $\lambda^6$ for $\lambda$, we may write the upper bound as \[ |G| \left( \inf_{0<\lambda<1 }  \frac{  \frac{p(p-1)(p-2)}{6} + p^2 \lambda^2 +  \frac{ p (p^2-7)}{6} \lambda^3 +\frac{p(p-1)(p-2)}{6} \lambda^6  }{  \frac{p(p+1)(p-1)}{2} \lambda^2  } \right)^n.\]
 
 Simplifying, we obtain  \[ |G| \left( \inf_{0<\lambda<1 }  \frac{  (p-1)(p-2)+ 6p\lambda^2 +    (p^2-7)\lambda^3 +(p-1)(p-2) \lambda^6  }{ 3(p+1)(p-1)\lambda^2  } \right)^n\]
 \[= |G|\inf_{0<\lambda<1 }   \left( \frac{ (p-2)+  6\lambda^2 + (p+1) \lambda^3 +(p-2) \lambda^6  }{3(p+1) \lambda^2   }  + \frac{2-2\lambda}{p^2-1}\right)^n\]\[ = |G| \inf_{\lambda \in (0,1)} \left( \frac{1 + \lambda^3 +\lambda^6}{3\lambda^2} + \frac{2 \lambda^2 - 1 - \lambda^6}{(p+1)\lambda^2} +\frac{2-2\lambda}{p^2-1} \right)^n   .\]

 \end{proof}

\begin{remark} \[\lim_{p \to \infty} \left( \inf_{\lambda \in (0,1)} \left( \frac{1 + \lambda^3 +\lambda^6}{3\lambda^2} + \frac{2 \lambda^2 - 1 - \lambda^6}{(p+1)\lambda^2} +\frac{2-2\lambda}{p^2-1} \right)\ \right) =\inf_{\lambda \in (0,1)} \left( \frac{1 + \lambda^3 +\lambda^6}{3\lambda^2} \right)\approx .919\]

This constant appears also in the Ellenberg-Gijswijt bound for sets in $\mathbb F_3^n$ free of $3$-term progressions, which is $\left(3 e^{-I(2/3)}\right)^n$ \cite[Corollary 5]{EG}. By performing the change of variables $\lambda=e^{3\theta}$ on \cite[Equation (2)]{EG}, one sees that $e^{-I(2/3)}=\inf_{\lambda \in (0,1)} \left( \frac{1 + \lambda^3 +\lambda^6}{3\lambda^2} \right)$. 

We do not know if this coincidence has any further significance. \end{remark}

\bibliographystyle{abstract}

\bibliography{references}

\end{document}